\documentclass[11pt]{article}

\usepackage[T1]{fontenc}
\usepackage{lmodern}
\usepackage{microtype}
\usepackage[a4paper,margin=29mm]{geometry}
\usepackage{amsmath,amssymb,amsthm,mathtools,mathrsfs}
\usepackage{enumitem}
\usepackage{needspace}
\usepackage{hyperref}
\usepackage[capitalise,noabbrev]{cleveref}
\usepackage{etoolbox}

\AtBeginEnvironment{proposition}{\crefalias{theorem}{proposition}}
\AtBeginEnvironment{lemma}{\crefalias{theorem}{lemma}}
\AtBeginEnvironment{corollary}{\crefalias{theorem}{corollary}}
\AtBeginEnvironment{definition}{\crefalias{theorem}{definition}}
\AtBeginEnvironment{hypothesis}{\crefalias{theorem}{hypothesis}}
\AtBeginEnvironment{remark}{\crefalias{theorem}{remark}}
\AtBeginEnvironment{example}{\crefalias{theorem}{example}}

\hypersetup{
  pdftitle={Hecke Orbits on Globally Drinfeld-Uniformized Shimura Varieties Are p-adically Nowhere Dense},
  pdfauthor={Yu Fu},
  colorlinks=false,
  linktoc=page,
  pdfborder={0 0 1},
  linkbordercolor={1 0 0},
  citebordercolor={0 1 0},
  urlbordercolor={0 0 1}
}

\setlist{itemsep=3pt,topsep=5pt}

\newtheorem{theorem}{Theorem}[section]
\newtheorem{proposition}[theorem]{Proposition}
\newtheorem{lemma}[theorem]{Lemma}
\newtheorem{corollary}[theorem]{Corollary}
\theoremstyle{definition}
\newtheorem{definition}[theorem]{Definition}

\newtheorem{remark}[theorem]{Remark}

\numberwithin{equation}{section}

\newcommand{\Cp}{\mathbb C_p}
\newcommand{\Qp}{\mathbb Q_p}
\newcommand{\Fpbar}{\overline{\mathbf F}_p}
\newcommand{\Adele}{\mathbf A}
\newcommand{\PP}{\mathbf P}
\newcommand{\HO}{\operatorname{HO}}
\newcommand{\Int}{\operatorname{Int}}
\newcommand{\red}{\operatorname{red}}
\newcommand{\an}{\mathrm{an}}
\newcommand{\Ber}{\mathrm{Ber}}

\newcommand{\GL}{\operatorname{GL}}
\newcommand{\PGL}{\operatorname{PGL}}
\newcommand{\PGU}{\operatorname{PGU}}
\newcommand{\Res}{\operatorname{Res}}

\title{Hecke Orbits on Globally Drinfeld-Uniformized Shimura Varieties Are $p$-adically Nowhere Dense}
\author{Yu Fu}
\date{}

\begin{document}
\maketitle
\makeatletter
\begingroup
\renewcommand{\thefootnote}{}
\long\def\@makefntext#1{\noindent#1}
\endgroup
\makeatother

\begin{abstract}
Fix a projective unitary Shimura datum satisfying Varshavsky's global Drinfeld uniformization hypotheses at a place $v\mid p$ of its reflex field $E$. Let $\mathcal{S}h$ be the base change of the Shimura variety to $\Cp$ at a sufficiently small neat finite level. We prove that for every closed reduced algebraic subvariety $Z\subset\mathcal{S}h$ with $\dim Z<\dim\mathcal{S}h$, the $p$-adic closure of its full finite adelic Hecke orbit has empty interior in $\mathcal{S}h(\Cp)$. At certain compatible levels, the nowhere density theorem also holds for Shimura varieties associated with a related PEL datum having the same adjoint datum as Varshavsky's.

Moreover, for every rational prime $p$ and every integer $d\geq2$, there is a projective unitary Shimura variety of dimension $d-1$ satisfying these uniformization hypotheses at a place above $p$ for which the nowhere density theorem holds.

\end{abstract}

\setcounter{tocdepth}{1}
\tableofcontents

\section{Introduction}

Hecke correspondences on a Shimura variety send algebraic subvarieties to algebraic subvarieties. Let $\mathcal{S}h$ be the base change of a Shimura variety at a fixed neat finite level to $\Cp$. For a closed algebraic subvariety $Z\subset\mathcal{S}h$, its full finite adelic Hecke orbit $\HO(Z)$ is the union of the $\Cp$-points of its images under all finite adelic Hecke correspondences. The finite level definition is given in \Cref{sec:definitions}.

A folklore conjecture predicts that if $\mathcal{S}h$ is a Shimura variety of Hodge type and $\dim Z<\dim\mathcal{S}h$, the $p$-adic closure of $\HO(Z)$ has empty interior in $\mathcal{S}h(\Cp)$; see \cite[Question~1.2 and Conjecture~1.3]{FuNowhereDensity}.

We prove the same nowhere density statement for Varshavsky's globally Drinfeld-uniformized projective unitary Shimura varieties and for the related PEL varieties described below.

Let $F$ be a totally real number field, let $K/F$ be a CM extension, and let $D^{\mathrm{int}}$ be a central simple algebra of degree $d\geq2$ over $K$ with an involution $\alpha^{\mathrm{int}}$ of the second kind. Assume that the associated unitary group has signature $(d-1,1)$ at $r\geq1$ real places and is definite at the others. We use the Shimura datum for $\mathbf H^{\mathrm{int}}=\Res_{F/\mathbb Q}\operatorname{GU}(D^{\mathrm{int}},\alpha^{\mathrm{int}})$ defined in \cite[\S2.2]{VarshavskyII}, and denote its reflex field by $E$.

Let $v$ be a finite place of $E$, let $p$ be its residue characteristic, and fix an embedding $E_v\hookrightarrow\Cp$. Let $w_1,\ldots,w_r$ be the places of $K$ above $p$ determined by the chosen Shimura datum and the reflex place $v$, as in \cite[\S2.5]{VarshavskyII}, and put $v_i=w_i|_F$. We impose the following conditions on $v$: the $v_i$ are distinct, each $v_i$ splits in $K$, and $\operatorname{inv}_{w_i}D^{\mathrm{int}}=1/d$ for every $i$. Under these hypotheses, Varshavsky's theorem gives the global Drinfeld uniformization recalled in \Cref{sec:varshavsky}.

\medskip
\begin{theorem}
\label{thm:main-intro}
For the Shimura datum and the place $v$ satisfying the hypotheses above, let $\mathcal{S}h$ be the base change to $\Cp$ of the entire Shimura variety at any sufficiently small neat finite level. Then $\dim\mathcal{S}h=r(d-1)$. For every closed reduced algebraic subvariety $Z\subset\mathcal{S}h$ with $\dim Z<\dim\mathcal{S}h$, the $p$-adic closure of its full finite adelic Hecke orbit has empty interior in $\mathcal{S}h(\Cp)$.
\end{theorem}

The theorem applies to every closed reduced subvariety of positive codimension. Such a subvariety may be Hodge generic or zero-dimensional. 
The adjoint group of these Shimura data is $\mathbb Q$-simple; see \Cref{rmk:Q-simple}. Thus even when the period domain is a product, the theorem is not obtained from a product decomposition of the adjoint Shimura datum.




\medskip
For a fixed Shimura datum, the hypotheses of \Cref{thm:main-intro} restrict the choice of $v$. Nevertheless, for every rational prime $p$ and every integer $d\geq2$, one can choose a Varshavsky datum of degree $d$ satisfying these hypotheses at a place above $p$; see \Cref{prop:one-place-existence}.

\medskip
Varshavsky's original Shimura data are not of Hodge type. In \Cref{subsec:pel-central-lift}, we construct a PEL Shimura datum $(\mathbf H^\sharp,\mathcal X^\sharp)$ with the same adjoint datum and  prove the following consequence of \Cref{thm:main-intro} for its Shimura varieties at the compatible levels of \Cref{lem:pel-central-lift-finite}.

\begin{corollary}
\label{cor:pel-intro}
Let $S^\sharp=\operatorname{Sh}_{\mathcal U^\sharp}(\mathbf H^\sharp,\mathcal X^\sharp)_{\Cp}$, where $\mathcal U^\sharp$ is a sufficiently small neat compact open subgroup satisfying the compatibility condition in \Cref{lem:pel-central-lift-finite}. For every closed reduced algebraic subvariety $Z^\sharp\subset S^\sharp$ with $\dim Z^\sharp<\dim S^\sharp$, the $p$-adic closure of its full finite adelic Hecke orbit is nowhere dense in $S^\sharp(\Cp)$.
\end{corollary}

The construction also shows that Varshavsky's data are of abelian type; see \Cref{prop:pel-datum}.

\paragraph{Strategy of the proof.}
Under Varshavsky's hypotheses at a finite place, every geometric component of the fixed level projective Shimura variety is a finite \'etale cover of an arithmetic quotient of a product of Drinfeld upper half spaces \cite[Proposition~2.8, Lemma~2.9, Remarks~2.11--2.12, and First Main Theorem~2.13]{VarshavskyII}.  The fact that this description covers every component at the fixed level is essential.

The analytic input is \Cref{thm:local-drinfeld} which is proved in \cite[Theorem~4.1]{FuLargeOrbits}. For a product $\mathcal D$ of Drinfeld upper half spaces and its associated group $J$, if $A$ is a reduced quasi-compact separated strict $\Cp$-analytic space and $f:A\to\mathcal D$ is analytic with $\dim A<\dim\mathcal D$, then $Jf(A(\Cp))$ is closed and nowhere dense in $\mathcal D(\Cp)$. The precise statement and notation are recalled in \Cref{sec:drinfeld}.

We combine Varshavsky's uniformization with the local theorem as follows. Since $\mathcal{S}h$ is projective and $Z$ is closed, $Z^{\Ber}$ is compact. We therefore choose finitely many quasi-compact analytic subspaces in the inverse image of $Z$ under the uniformization whose arithmetic translates cover that inverse image. These analytic subspaces are chosen independently of the Hecke correspondence. Varshavsky's full finite adelic equivariance shows that the period image of every lift of a point of $\HO(Z)$ lies in a $J$-translate of the period image of one of these subspaces, while the local theorem makes the union of these $J$-translates closed and nowhere dense. We pull this union back along the finite \'etale period maps and descend it through the open arithmetic quotient maps. Thus the argument controls all Hecke correspondences simultaneously, rather than treating them one at a time and taking a countable union. We formulate this argument as a general criterion in \Cref{sec:global-criterion} and verify its period containment hypothesis for all finite adelic Hecke correspondences in \Cref{prop:full-period containment}.

The criterion in \Cref{thm:global-criterion} is stated for more general period domains and only requires the period maps to be open on $\Cp$-points. A local nowhere density statement for orbit closures and the full period containment property remain hypotheses. \Cref{prop:uniform-period-reduction} gives an alternative to compactness by imposing a reduction condition that controls the periods of all lifts at once. 

\paragraph{Related work.}
In the proper good-reduction setting discussed in \Cref{subsec:maulik-poonen}, Maulik--Poonen's theorem \cite{MaulikPoonen} gives related nowhere density statements for points and for positive-dimensional subvarieties with generically extra rational endomorphisms.  By contrast, our theorem imposes no such condition on $Z$: it applies to every closed reduced subvariety of positive codimension including Hodge-generic ones under Varshavsky's global uniformization hypotheses.

\paragraph{Organization of the paper.}
\Cref{sec:definitions} fixes the topology and Hecke notation. \Cref{sec:global-criterion} proves the global period containment criterion. \Cref{sec:varshavsky} recalls the local Drinfeld theorem, verifies the criterion in Varshavsky's family, and proves the main theorem, first at product levels and then at arbitrary sufficiently small neat levels.  \Cref{subsec:pel-central-lift} constructs the related PEL datum and proves its Hecke-orbit consequence at compatible levels.  \Cref{sec:scope} compares our result with Maulik--Poonen's theorem and discusses Kim's uniformization and the difficulties in extending our argument to general Shimura varieties.

\section{Preliminaries}
\label{sec:definitions}

All closures in the statements are taken in the $p$-adic topology on $\Cp$-points. For a smooth $\Cp$-analytic space $X$ and a subset $S\subset X(\Cp)$, we write $\overline S^{\,\mathrm{cl}}$ for this closure. We say that $S$ is \emph{nowhere dense} in the $p$-adic topology if $\Int_{X(\Cp)}\overline S^{\,\mathrm{cl}}=\varnothing$. We introduce Berkovich spaces for compactness and analytic coverings.

Let $\mathbf H$ be the reductive $\mathbb Q$-group in a Shimura datum, and fix an embedding of its reflex field into $\Cp$. For a neat compact open subgroup $K\subset\mathbf H(\Adele_f)$, let $\mathcal{S}h_K$ denote the base change to $\Cp$ of the associated Shimura variety at level $K$. For $g\in\mathbf H(\Adele_f)$, let $K_g=K\cap gKg^{-1}$.

\begin{definition}
The \emph{finite adelic Hecke correspondence associated with $g$} is the diagram
\[
\mathcal{S}h_K
\xleftarrow{\ p_{1,g}\ }
\mathcal{S}h_{K_g}
\xrightarrow{\ p_{2,g}\ }
\mathcal{S}h_K,
\]
where $p_{1,g}$ is induced by $K_g\subset K$, and $p_{2,g}$ is induced by right translation by $g$ from level $K_g$ to level $g^{-1}K_g g$, followed by the change of level map for $g^{-1}K_g g\subset K$. Since $K$ is neat, both maps are finite \'etale.

For a closed subvariety $Z\subset\mathcal{S}h_K$, let $T_g(Z)\subset\mathcal{S}h_K$ be the reduced scheme-theoretic image of $p_{1,g}^{-1}(Z)$ under $p_{2,g}$. The \emph{full finite adelic Hecke orbit of $Z$ at level $K$} is $\HO_K(Z)=\bigcup_{g\in\mathbf H(\Adele_f)}T_g(Z)(\Cp)$.
\end{definition}

When the level is fixed, we write $\mathcal{S}h$ and $\HO(Z)$. The word \emph{full} means that $g$ ranges over all of $\mathbf H(\Adele_f)$; in particular, we also allow Hecke correspondences at $p$. These correspondences come from the standard right action on the Shimura tower; see \cite[Remark~5.29 and Theorem~13.6]{MilneShimura}.

A strict $\Cp$-analytic space admits an affinoid covering by spectra of strict $\Cp$-affinoid algebras. A quasi-compact separated strict analytic space has compact Berkovich space; see \cite[\S1.2]{BerkovichSpectral} for compactness of affinoid Berkovich spectra. A finite flat analytic map is open \cite[Proposition~3.2.7]{BerkovichEtale}. More generally, an étale map between smooth $\Cp$-analytic spaces induces a local analytic isomorphism on $\Cp$-points, by the nonarchimedean inverse function theorem, and is therefore open in the $p$-adic topology. By an analytic covering we mean a surjective analytic map such that every point of the target Berkovich space has an open neighborhood whose inverse image is a disjoint union of open analytic domains mapping isomorphically to that neighborhood.

\section{Local input and a global period containment criterion}
\label{sec:global-criterion}

We combine the local theorem with descent through period maps and arithmetic quotients. When $Z^{\Ber}$ is compact, finitely many quasi-compact lifts give one closed nowhere dense set containing the periods of all lifted targets. At the end of the section, we give another sufficient condition that does not require compactness.

\subsection{The global setup and period containment}
\label{subsec:criterion-setup}

We first give a general criterion, which we will apply to Varshavsky's Shimura varieties.

Let $\mathcal D$ be a smooth separated strict $\Cp$-analytic space of pure dimension $n\geq1$, and let $J$ be a topological group acting on $\mathcal D$ by analytic automorphisms. Let $C$ be a finite set. For every $c\in C$, assume that we have
\begin{enumerate}[label=\textup{(\arabic*)}]
\item a smooth separated strict $n$-dimensional $\Cp$-analytic space $\mathcal M_c$, a discrete torsion-free subgroup $\Gamma_c\subset J$ acting freely and properly discontinuously on $\mathcal M_c$, and the geometric analytic quotient $q_c:\mathcal M_c\to X_c=\Gamma_c\backslash\mathcal M_c$;
\item an analytic map $\pi_c:\mathcal M_c\to\mathcal D$ whose induced map on $\Cp$-points is open in the $p$-adic topology, satisfying
\begin{equation}
\label{eq:criterion-equivariance}
 \pi_c(\gamma m)=\gamma\pi_c(m)
\end{equation}
for \(\gamma\in\Gamma_c\);
\item an identification of an open analytic subspace of the analytification of a finite level Shimura variety with $\coprod_{c\in C}X_c$.
\end{enumerate}
The maps $q_c$ are surjective analytic coverings whose fibers on Berkovich spaces are $\Gamma_c$-orbits. They induce continuous open surjections on $\Cp$-points.

For \Cref{thm:global-criterion}, we assume the following local hypothesis: \textbf{for every reduced strict affinoid $\Cp$-analytic space $A$ with $\dim A<n$ and every analytic map $f:A\to\mathcal D$, the $p$-adic closure of $Jf(A(\Cp))$ is nowhere dense in $\mathcal D(\Cp)$.}

\begin{definition}
\label{def:period containment}
Let $Z\subset\coprod_cX_c$ be a closed reduced analytic subspace, and let $W_c=q_c^{-1}(Z\cap X_c)$. A \emph{family of finite correspondences} on $\coprod_cX_c$ means a family of diagrams
\[
\coprod_{c\in C}X_c\xleftarrow{\ s_h\ }Y_h\xrightarrow{\ t_h\ }\coprod_{c\in C}X_c
\]
with $s_h$ and $t_h$ finite. It has the \emph{full period containment property for $Z$} if the following pointwise condition holds. For every $h$, every $y\in Y_h(\Cp)$ with $s_h(y)\in Z$, and every lift $m'\in\mathcal M_{c'}(\Cp)$ of $t_h(y)$, there are an index $c$, a point $m\in W_c(\Cp)$, and an element $j\in J$ such that
\begin{equation}
\label{eq:abstract-period containment}
 \pi_{c'}(m')=j\pi_c(m).
\end{equation}
 The element $j$ may depend on $h$, $y$, and the chosen lift $m'$. Thus the period of every lifted target point lies in $\bigcup_{c\in C}J\pi_c(W_c(\Cp))$.
\end{definition}

\begin{theorem}
\label{thm:global-criterion}
Assume \textup{(1)--(3)} and the local hypothesis in \Cref{subsec:criterion-setup}. Let $Z\subset\coprod_cX_c$ be a closed reduced analytic subspace such that $Z^{\Ber}$ is compact, and suppose that $\dim Z<n$. Let $\mathscr H$ be a family of finite correspondences with the full period containment property for $Z$. Then the $p$-adic closure of the set of all $t_h(y)$ with $h\in\mathscr H$, $y\in Y_h(\Cp)$, and $s_h(y)\in Z$ is nowhere dense in $\coprod_cX_c(\Cp)$.
\end{theorem}

\subsection{Proof of the criterion}

We first choose the finite quasi-compact lifts and then record how a closed nowhere dense $J$-invariant set descends through the period maps and arithmetic quotients.

\begin{lemma}
\label{lem:finite-qc-lifts}
Fix $c\in C$, and let $Z\subset X_c$ be a closed reduced analytic subspace such that $Z^{\Ber}$ is compact. Let $s=\dim Z$. Then there are finitely many reduced quasi-compact strict analytic subspaces $A_1,\ldots,A_t\subset q_c^{-1}(Z)$ such that \(\dim A_i\leq s,\) and  \(q_c^{-1}(Z)\subseteq \Gamma_c\cdot\bigcup_{i=1}^tA_i.\)

\end{lemma}

\begin{proof}
For every $z\in Z^{\Ber}$, choose an open neighborhood $U_z\subset X_c^{\Ber}$ that is evenly covered by $q_c^{\Ber}$. Since $X_c$ is an open analytic subspace of the analytification of an algebraic variety, we can choose a strict affinoid domain $V_z\subset X_c$ such that $z\in\operatorname{Int}_{X_c^{\Ber}}(V_z^{\Ber})$ and $V_z^{\Ber}\subset U_z$.
As $z$ ranges over $Z^{\Ber}$, the sets $\operatorname{Int}_{X_c^{\Ber}}(V_z^{\Ber})$ form an open cover of $Z^{\Ber}$. By compactness, finitely many of them cover $Z^{\Ber}$. Denote the corresponding affinoid domains by $V_1,\ldots,V_t$ and their evenly covered neighborhoods by $U_1,\ldots,U_t$. In particular, $Z^{\Ber}\subseteq\bigcup_{i=1}^tV_i^{\Ber}$.
For each $i$, choose one sheet over $U_i$, and let $\widetilde V_i\subset\mathcal M_c$ be the affinoid domain in that sheet mapping isomorphically to $V_i$. Let $A_i=q_c^{-1}(Z)\cap\widetilde V_i$.

The restriction of $q_c$ identifies $A_i$ with $Z\cap V_i$. Since $Z$ is a closed analytic subspace of $X_c$, the intersection $Z\cap V_i$ is a closed analytic subspace of the strict affinoid space $V_i$, and is therefore strict affinoid. Thus $A_i$ is quasi-compact and separated. Moreover, $Z\cap V_i$ is an analytic domain in the reduced space $Z$, so it is reduced and has dimension at most $\dim Z$. The same properties therefore hold for $A_i$.

Take $m\in(q_c^{-1}(Z))^{\Ber}$. Some $V_i^{\Ber}$ contains $q_c^{\Ber}(m)$. Since $\widetilde V_i\to V_i$ is an isomorphism, there is a point $a\in\widetilde V_i^{\Ber}$ with $q_c^{\Ber}(a)=q_c^{\Ber}(m)$. We have $a\in A_i^{\Ber}$ because $q_c^{\Ber}(m)\in Z^{\Ber}$. Since fibers of $q_c^{\Ber}$ are $\Gamma_c$-orbits, we have $m=\gamma a$ for some $\gamma\in\Gamma_c$. This proves the lemma.
\end{proof}

\begin{lemma}
\label{lem:pullback-quotient}
Let $B\subset\mathcal D(\Cp)$ be a set which is closed, $J$-invariant, and nowhere dense. For every $c$, the set $\pi_c^{-1}(B)$ is closed and nowhere dense in $\mathcal M_c(\Cp)$. It is also $\Gamma_c$-invariant, and its image in $X_c(\Cp)$ is closed and nowhere dense.
\end{lemma}

\begin{proof}
The inverse image $\pi_c^{-1}(B)$ is closed because $\pi_c$ is continuous. By assumption, the induced map on $\Cp$-points is open in the $p$-adic topology. If $\pi_c^{-1}(B)$ contained a nonempty open set, its image would be a nonempty open subset of $B$. Thus $\pi_c^{-1}(B)$ is nowhere dense. For $m\in\pi_c^{-1}(B)$ and $\gamma\in\Gamma_c$, equivariance \eqref{eq:criterion-equivariance} gives $\pi_c(\gamma m)=\gamma\pi_c(m)\in B$, since $\Gamma_c\subset J$ and $B$ is $J$-invariant. Hence $\pi_c^{-1}(B)$ is $\Gamma_c$-invariant.

The fibers of $q_c$ are $\Gamma_c$-orbits, so this invariance gives $q_c^{-1}(q_c(\pi_c^{-1}(B)))=\pi_c^{-1}(B)$. Consequently, the complement of $q_c(\pi_c^{-1}(B))$ in $X_c(\Cp)$ is $q_c(\mathcal M_c(\Cp)\setminus\pi_c^{-1}(B))$. This complement is open because $q_c$ is an open map. Thus $q_c(\pi_c^{-1}(B))$ is closed. Suppose that it contains a nonempty open subset $U$. Since $q_c$ is continuous and surjective, $q_c^{-1}(U)$ is a nonempty open subset of $\mathcal M_c(\Cp)$. Then $q_c^{-1}(q_c(\pi_c^{-1}(B)))=\pi_c^{-1}(B)$ implies $q_c^{-1}(U)\subseteq\pi_c^{-1}(B)$, contradicting the nowhere density of $\pi_c^{-1}(B)$.
\end{proof}

\begin{proof}[Proof of \Cref{thm:global-criterion}]
For every $c\in C$, the space $(Z\cap X_c)^{\Ber}$ is a closed subset of the compact space $Z^{\Ber}$, hence compact. Apply \Cref{lem:finite-qc-lifts} to $Z\cap X_c$ for every $c\in C$. Since $C$ is finite, after relabeling we obtain finitely many spaces $A_1,\ldots,A_t$ and indices $c(1),\ldots,c(t)$ such that $A_i\subset W_{c(i)}$, $\dim A_i\leq\dim Z$, and, for every $c\in C$, 
\begin{equation}\label{eq:lemma3.4}
    W_c(\Cp)\subseteq\Gamma_c\cdot\bigcup_{i:c(i)=c}A_i(\Cp).
\end{equation} The spaces $A_i$ constructed in the lemma are reduced strict affinoid spaces, hence quasi-compact and separated. Since $\dim A_i<n$, the local hypothesis applied to $\pi_{c(i)}|_{A_i}$ shows that $B_i=\overline{J\pi_{c(i)}(A_i(\Cp))}^{\,\mathrm{cl}}$ is closed and nowhere dense in $\mathcal D(\Cp)$. Each element of $J$ acts by a homeomorphism, so $B_i$ is also $J$-invariant. The finite union $B=\bigcup_{i=1}^tB_i$ is also closed, $J$-invariant, and nowhere dense.

Let $m'$ be a lift of $t_h(y)$ for some $h\in\mathscr H$ and some $y\in Y_h(\Cp)$ with $s_h(y)\in Z$, and let $c'$ be the index for which $m'\in\mathcal M_{c'}$. By the period containment property, there are $c\in C$, $m\in W_c(\Cp)$, and $j\in J$ such that $\pi_{c'}(m')=j\pi_c(m)$. By \eqref{eq:lemma3.4}, there are $i$ with $c(i)=c$, a point $a\in A_i(\Cp)$, and $\gamma\in\Gamma_c$ such that $m=\gamma a$. Equivariance gives $\pi_c(m)=\gamma\pi_c(a)$. Since $\Gamma_c\subset J$, we obtain $\pi_{c'}(m')=j\gamma\pi_c(a)\in B_i\subset B$. Thus every lifted target point lies in $\pi_{c'}^{-1}(B)$ for the relevant $c'$, and its image lies in $q_{c'}(\pi_{c'}^{-1}(B))$. By \Cref{lem:pullback-quotient}, these images are closed and nowhere dense in their respective components. Their finite union is closed and nowhere dense in $\coprod_cX_c(\Cp)$ and contains all the target points above, hence also their $p$-adic closure. This proves the theorem.
\end{proof}

Under the same local hypothesis, compactness of $Z^{\Ber}$ may also be replaced by the existence of finitely many reduced strict affinoid spaces $A_i$ with $\dim A_i<n$ and analytic maps $f_i:A_i\to\mathcal D$ such that $\bigcup_cJ\pi_c(W_c(\Cp))\subseteq\bigcup_i\overline{Jf_i(A_i(\Cp))}^{\,\mathrm{cl}}$. The finite union on the right is closed, $J$-invariant, and nowhere dense. Period containment and \Cref{lem:pullback-quotient} give the same conclusion. The spaces $A_i$ need not cover the lifted subvariety.

\medskip
We prove the following proposition under a uniform reduction condition.

\begin{proposition}
\label{prop:uniform-period-reduction}
Assume \textup{(1)--(3)}. Let $Z\subset\coprod_cX_c$ be a closed reduced analytic subspace, let $W_c=q_c^{-1}(Z\cap X_c)$, and let $\mathscr H$ be a family of finite correspondences with the full period containment property for $Z$. Let $T=\bigcup_{c\in C}J\pi_c(W_c(\Cp))$. For every strict closed rational coordinate polydisc $V\subset\mathcal D$, translate and rescale its coordinates to identify it with the unit polydisc, and let $\red_V:V(\Cp)\to\mathbb A^n(\Fpbar)$ be coordinatewise reduction. Suppose that for every such $V$ there is a Zariski closed subset $\mathcal T_V\subset\mathbb A^n_{\Fpbar}$ with $\dim\mathcal T_V<n$ such that $\red_V(T\cap V(\Cp))\subseteq\mathcal T_V(\Fpbar)$. Then the conclusion of \Cref{thm:global-criterion} holds.
\end{proposition}

\begin{proof}
Let $B=\overline T^{\,\mathrm{cl}}$. Take any nonempty $p$-adic open subset $U\subset\mathcal D(\Cp)$ and a strict closed rational coordinate polydisc $V$ with $V(\Cp)\subset U$. Since $\dim\mathcal T_V<n$, choose $\bar z\in\mathbb A^n(\Fpbar)\setminus\mathcal T_V(\Fpbar)$. The reduction fiber $\red_V^{-1}(\bar z)$ is nonempty and open in $\mathcal D(\Cp)$. It is disjoint from $T$. Since it is open, it is also disjoint from $B$. Thus $U$ is not contained in $B$, proving that $B$ is nowhere dense.

The set $T$ is $J$-invariant. Since every element of $J$ acts by a homeomorphism, $B$ is also $J$-invariant. Period containment places every lifted target point in the corresponding $\pi_c^{-1}(B)$. By \Cref{lem:pullback-quotient}, the finite union $\bigcup_cq_c(\pi_c^{-1}(B))$ is closed and nowhere dense in $\coprod_cX_c(\Cp)$ and contains all target points, hence also their $p$-adic closure.
\end{proof}

Suppose that each $W_c$ has a quasi-compact exhaustion $W_c=\bigcup_{m\geq1}A_{c,m}$, and let $T_m=\bigcup_{c\in C}J\pi_c(A_{c,m}(\Cp))$. Then $T=\bigcup_{m\geq1}T_m$. This exhaustion alone does not give the required reduction condition. For each fixed target polydisc $V$, there must be a Zariski closed subset $\mathcal T_V\subset\mathbb A^n_{\Fpbar}$ with $\dim\mathcal T_V<n$ such that $\red_V(T_m\cap V(\Cp))\subseteq\mathcal T_V(\Fpbar)$ for every $m\geq1$. The subset $\mathcal T_V$ may depend on $V$, but not on $m$. Allowing a different low-dimensional closed subset $\mathcal T_{V,m}$ for each $m$ does not suffice, since $\bigcup_m\mathcal T_{V,m}$ may still be Zariski dense in $\mathbb A^n_{\Fpbar}$.

\section{Varshavsky's uniformization and the main theorem}
\label{sec:varshavsky}

In this section, we prove \Cref{thm:main-intro} by verifying the hypotheses of \Cref{thm:global-criterion} for the projective unitary Shimura varieties satisfying Varshavsky's uniformization hypotheses at the chosen place $v$. We first recall the local Drinfeld theorem and then describe the uniformization and the group actions used in the proof.

\subsection{The local theorem on a product of Drinfeld spaces}
\label{sec:drinfeld}

The following theorem is proved in \cite[Theorem~4.1]{FuLargeOrbits}. Fix $r\geq1$. For every $i$, let $d_i\geq2$ and let $F_i/\Qp$ be a finite extension equipped with an embedding $F_i\hookrightarrow\Cp$. 
Write $\Omega_{F_i}^{d_i-1}$ for Drinfeld's upper half space which is  the complement in $(\PP^{d_i-1}_{F_i})^{\an}$ of the analytifications of all $F_i$-rational hyperplanes. Let
\[
 \mathcal D=\prod_{i=1}^r(\Omega_{F_i}^{d_i-1})_{\Cp},
 \quad
 J=\prod_{i=1}^r\PGL_{d_i}(F_i),
 \quad
 n=\dim\mathcal D=\sum_{i=1}^r(d_i-1).
\]

\begin{theorem}
\label{thm:local-drinfeld}
Let $A$ be a reduced, quasi-compact, separated, strict $\Cp$-analytic space, and let $f:A\to\mathcal D$ be analytic. If $\dim A<n$, then $Jf(A(\Cp))$ is closed and nowhere dense in $\mathcal D(\Cp)$.
\end{theorem}

The proof combines the properness of the $J$-action on $\mathcal D^{\Ber}$ with a finite-reduction argument for compact subsets of $J$. Quasi-compactness makes $f(A^{\Ber})$ compact and provides a finite affinoid cover of $A$ for the reduction argument.

\subsection{The global datum and the period domain}
We keep the unitary datum and the place $v$ fixed in the introduction, and let $G^{\mathrm{int}}=\operatorname{GU}(D^{\mathrm{int}},\alpha^{\mathrm{int}})$. Let $w_1,\ldots,w_r$ be the places of $K$ fixed in the introduction. For each $i$, let $\alpha_i:K_{w_i}\hookrightarrow E_v$ be the embedding constructed in \cite[\S2.5 and Lemma~2.6]{VarshavskyII}. By the same lemma, $E_v$ is the compositum of the subfields $\alpha_i(K_{w_i})$. All scalar extensions from $K_{w_i}$ below are taken through $\alpha_i$. Varshavsky chooses an inner form $G$ which is compact modulo its center at every real place and split at the places $v_i$; see \cite[\S2.7]{VarshavskyII}. At a uniformizing place the relevant local factors have the form
\begin{equation}
\label{eq:two-local-groups}
 \begin{aligned}
 G(F_{v_i}) \simeq
       \GL_d(K_{w_i})\times F_{v_i}^{\times},\quad
 G^{\mathrm{int}}(F_{v_i}) \simeq
       \widetilde D_i^{\times}\times F_{v_i}^{\times},
 \end{aligned}
\end{equation}
where $\widetilde D_i$ is the central division algebra over $K_{w_i}$ with Hasse invariant $1/d$. The extra factor is central and acts trivially on the projective period domain.

Let $\mathbf H^{\mathrm{int}}=\Res_{F/\mathbb Q}G^{\mathrm{int}}$ and $\mathbf H=\Res_{F/\mathbb Q}G$. Thus $\mathbf H^{\mathrm{int}}(\Adele_f)= G^{\mathrm{int}}(\Adele_{F,f})$. The Hecke correspondences below are those defined by elements of $\mathbf H^{\mathrm{int}}(\Adele_f)$. We use the $F$-group notation only when we decompose its local factors.

For each $i$, let $\Omega_i=\Omega_{K_{w_i}}^{d-1}$ be the complement of the analytifications of the $K_{w_i}$-rational hyperplanes in $(\PP^{d-1}_{K_{w_i}})^{\an}$. Let $\mathcal D=\prod_{i=1}^r(\Omega_i)_{\Cp}$ and $J=\prod_{i=1}^r\PGL_d(K_{w_i})$. These are the domain and group of \Cref{sec:drinfeld} with $F_i=K_{w_i}$ and $d_i=d$, so $\dim\mathcal D=r(d-1)$.

\subsection{The local actions and the double quotient}

For each $i$, let $\widehat{K_{w_i}^{\mathrm{nr}}}$ be the completion of the maximal unramified extension of $K_{w_i}$. Drinfeld's construction gives a sequence of finite étale covers
\[
\cdots\longrightarrow\Sigma_{K_{w_i}}^{d,\ell+1}\longrightarrow\Sigma_{K_{w_i}}^{d,\ell}\longrightarrow\cdots\longrightarrow\Sigma_{K_{w_i}}^{d,0},
\quad
\Sigma_{K_{w_i}}^{d,0}=\Omega_i\widehat\otimes_{K_{w_i}}\widehat{K_{w_i}^{\mathrm{nr}}},
\]
where the transition maps forget part of the level structure; see \cite[\S1.4.1]{VarshavskyI}. We denote this inverse system by $\Sigma_i$. Composing the maps to level zero with the projection $\Sigma_{K_{w_i}}^{d,0}\to\Omega_i$ gives compatible period projections, which we denote by $\rho_i:\Sigma_{K_{w_i}}^{d,\ell}\to\Omega_i$ at every level. The tower has commuting actions of $\GL_d(K_{w_i})$ and $\widetilde D_i^\times$, and these projections satisfy $\rho_i(gx)=\overline g\,\rho_i(x)$ and $\rho_i(\delta x)=\rho_i(x)$ for $g\in\GL_d(K_{w_i})$ and $\delta\in\widetilde D_i^\times$, where $\overline g$ is the image of $g$ in $\PGL_d(K_{w_i})$; see \cite[\S1.4.1]{VarshavskyI}. When the action of $\widetilde D_i^\times$ is written on the right, we use the convention $x\cdot\delta=\delta^{-1}x$. This right action commutes with the left $\GL_d(K_{w_i})$-action and leaves $\rho_i$ unchanged. After base change through $K_{w_i}\hookrightarrow E_v\hookrightarrow\Cp$, we write $\rho=\prod_i\rho_i$ for the product projection to $\mathcal D$ at each level.

In the notation of \cite[\S2.7]{VarshavskyII}, let $G'=\left(\prod_{i=1}^rF_{v_i}^{\times}\right)\times G (\Adele_{F,f}^{\{v_1,\ldots,v_r\}})$, where the superscript means that the factors at $v_1,\ldots,v_r$ are omitted from the finite adeles. Fix the identifications $G^{\mathrm{int}}(F_u)\simeq G(F_u)$ at the finite places $u\notin\{v_1,\ldots,v_r\}$ as in \cite[\S2.7]{VarshavskyII}. Together with \eqref{eq:two-local-groups}, these identifications give
\begin{equation}
\label{eq:global-adelic-decompositions}
\begin{aligned}
G(\Adele_{F,f})
  &\simeq \left(\prod_{i=1}^r\GL_d(K_{w_i})\right)\times G',\\
G^{\mathrm{int}}(\Adele_{F,f})
  &\simeq \left(\prod_{i=1}^r\widetilde D_i^\times\right)\times G'.
\end{aligned}
\end{equation}

Let $\mathcal O_{\widetilde D_i}$ be the maximal order of $\widetilde D_i$, and choose a uniformizer $\varpi_i$ of its center $K_{w_i}$. Following \cite[\S1.4.1]{VarshavskyI}, set $T_{0,i}=\mathcal O_{\widetilde D_i}^{\times}$ and $T_{\ell,i}=1+\varpi_i^\ell\mathcal O_{\widetilde D_i}$ for $\ell\geq1$. Each $T_{\ell,i}$ is normal in $\widetilde D_i^\times$. For $\ell'\geq\ell$, the transition map identifies $\Sigma_{K_{w_i}}^{d,\ell}$ with the quotient of $\Sigma_{K_{w_i}}^{d,\ell'}$ by $T_{\ell,i}/T_{\ell',i}$; see \cite[\S1.4.1]{VarshavskyI}.

Under the second decomposition in \eqref{eq:global-adelic-decompositions}, choose an integer $\ell\geq0$ and a compact open subgroup $U'\subset G'$ such that
\begin{equation}
\label{eq:varshavsky-product-level}
 \mathcal U=
 \left(\prod_{i=1}^rT_{\ell,i}\right)\times U'
 \subset G^{\mathrm{int}}(\Adele_{F,f})
\end{equation}
is sufficiently small and neat. Let $S=\operatorname{Sh}_{\mathcal U}(\mathbf H^{\mathrm{int}})$ denote the canonical model over $E$, and put $n=\dim S$. At this level, the variety in the introduction is $\mathcal{S}h=S_{\Cp}=S\times_E\Cp$, where the base change uses the fixed embedding $E\hookrightarrow E_v\hookrightarrow\Cp$. By \cite[Proposition~2.8, Lemma~2.9, \S2.10, and First Main Theorem~2.13]{VarshavskyII}, $S$ is smooth and projective and $n=r(d-1)=\dim\mathcal D$.

At the level $\mathcal U=(\prod_{i=1}^rT_{\ell,i})\times U'$ fixed in \eqref{eq:varshavsky-product-level}, the double quotient in \cite[\S2.7]{VarshavskyII} and the equivariant isomorphism of \cite[First Main Theorem~2.13]{VarshavskyII} give the following uniformization of $S_{E_v}^{\an}$. We write the $G(F)$-action on the left and the $U'$-action on the right, with $U'$ acting by right multiplication on the $G'$-coordinate:
\[
 S_{E_v}^{\an}\simeq G(F)\backslash \left[ \left(\prod_{i=1}^r \left( \Sigma_{K_{w_i}}^{d,\ell} \widehat\otimes_{K_{w_i},\,\alpha_i}E_v \right) \right)\times G' \right]/U'.
\]
If $\gamma=(\gamma_{\mathrm{unif}},\gamma')\in G(F)$ under the first decomposition in \eqref{eq:global-adelic-decompositions}, it acts by $\gamma\cdot(m,a)=(\gamma_{\mathrm{unif}}\cdot m,\gamma'a)$. For $(\delta,b)\in(\prod_i\widetilde D_i^\times)\times G'$, the right action is $(m,a)\cdot(\delta,b)=(m\cdot\delta,ab)$. These actions commute. The formulas are written before taking the quotient by $U'$. The full finite adelic group acts on the tower; at a fixed level its elements give Hecke correspondences, not in general automorphisms of $S$.

\subsection{Finite-level component presentations}

Let $C$ index the connected components of $\mathcal{S}h$, and let $X_c$ be the analytification of the component indexed by $c \in C$. The set $C$ is finite because $\mathcal{S}h$ is of finite type. In the double quotient, the class of the $G'$-coordinate in $G(F)\backslash G'/U'$ is locally constant, where $G(F)$ acts on $G'$ through $\gamma\mapsto\gamma'$. It is therefore constant on $X_c$. Choose a representative $a_c\in G'$ of this double coset. Different geometric components may have the same $a_c$.

By \cite[Remark~2.12]{VarshavskyII}, each $X_c$ has a finite \'etale map $f_c:X_c\to\Gamma_c\backslash\mathcal D$ for a torsion-free cocompact arithmetic lattice $\Gamma_c\subset J$. We choose $\Gamma_c$ and $f_c$ from the quotient construction associated with $a_c$ in \cite[\S2.7 and the proof of Proposition~2.8]{VarshavskyII}. Every point $x\in X_c(\Cp)$ has a representative $(m,a)$ with $a\in a_cU'$. For such a representative, the construction gives $f_c(x)=[\rho(m)]$, where the brackets denote the $\Gamma_c$-orbit. Define
\[
 \mathcal M_c=
 \mathcal D\times_{\Gamma_c\backslash\mathcal D}X_c,
 \quad
 q_c:\mathcal M_c\longrightarrow X_c,
 \quad
 \pi_c:\mathcal M_c\longrightarrow\mathcal D.
\]
Here $q_c$ and $\pi_c$ are the second and first projections, respectively. The group $\Gamma_c$ acts on $\mathcal M_c$ through the first factor, $\Gamma_c\backslash\mathcal M_c\simeq X_c$, and $\pi_c$ is finite \'etale and $\Gamma_c$-equivariant. If $(m,a)$ with $a\in a_cU'$ represents $x\in X_c(\Cp)$, then $(\rho(m),x)\in\mathcal M_c(\Cp)$ is a lift of $x$. Two lifts under $q_c$ differ by an element of $\Gamma_c\subset J$.

After choosing the finite set of adelic and geometric representatives for the fixed finite level, the analytic uniformization reads
\begin{equation}
\label{eq:component-uniformization}
 \mathcal{S}h^{\an}
 \simeq\coprod_{c\in C}X_c
 =\coprod_{c\in C}\Gamma_c\backslash\mathcal M_c.
\end{equation}
This decomposes the whole fixed level variety, and $C$ depends only on $\mathcal U$, not on a Hecke element.

\subsection{The full finite adelic period containment}

The tower isomorphism in \cite[Theorem~2.13]{VarshavskyII} is equivariant for $\mathbf H^{\mathrm{int}}(\Adele_f) =G^{\mathrm{int}}(\Adele_{F,f})$. Let $Z\subset \mathcal{S}h$ be a closed reduced algebraic subvariety. For each $c\in C$, put $W_c=q_c^{-1}(Z^{\an}\cap X_c)\subset\mathcal M_c$.

\begin{proposition}
\label{prop:full-period containment}
For every closed reduced $Z\subset \mathcal{S}h$, the full finite adelic Hecke correspondences satisfy the full period containment property for $Z^{\an}$ with respect to the component presentations in \eqref{eq:component-uniformization}.
\end{proposition}

\begin{proof}
Fix $h\in\mathbf H^{\mathrm{int}}(\Adele_f)$. Under the second decomposition in \eqref{eq:global-adelic-decompositions}, write $h=(\delta,h')$ with $\delta\in\prod_i\widetilde D_i^\times$ and $h'\in G'$. Let $\mathcal U_h=\mathcal U\cap h\mathcal U h^{-1}$. Since every $T_{\ell,i}$ is normal in $\widetilde D_i^\times$, we have $\mathcal U_h=(\prod_iT_{\ell,i})\times(U'\cap h'U'h'^{-1})$. As in \Cref{sec:definitions}, the map $s_h: \operatorname{Sh}_{\mathcal U_h}(\mathbf H^{\mathrm{int}}) \to S$ is the change of level map, and the map $t_h: \operatorname{Sh}_{\mathcal U_h}(\mathbf H^{\mathrm{int}}) \to S$ is induced by right translation by $h$.

Let $y\in \operatorname{Sh}_{\mathcal U_h}(\mathbf H^{\mathrm{int}})(\Cp)$ satisfy $s_h(y)\in Z$, and let $c$ be the component containing $s_h(y)$. Represent $y$ at level $\mathcal U_h$ by a pair $(\widetilde m,a)$, where $\widetilde m\in (\prod_{i=1}^r(\Sigma_{K_{w_i}}^{d,\ell}\widehat\otimes_{K_{w_i},\,\alpha_i}\Cp))(\Cp)$ and $a\in G'$. Replacing the pair by a $G(F)$-equivalent representative allows us to assume $a\in a_cU'$.  Since $f_c(s_h(y))=[\rho(\widetilde m)]$, the point $m=(\rho(\widetilde m),s_h(y))$ belongs to $W_c(\Cp)$ and lies above $s_h(y)$.

The map $s_h$ forgets part of the level structure, so $s_h(y)$ is represented at level $\mathcal U$ by the same pair $(\widetilde m,a)$. The map $t_h$ is induced by the translation by $h=(\delta,h')$, so $t_h(y)$ is represented by $(\widetilde m\cdot\delta,ah')$. Since the action of $\prod_i\widetilde D_i^\times$ leaves the period projection unchanged, we have $\rho(\widetilde m\cdot\delta)=\rho(\widetilde m)$.

Let $c'$ be the component containing $t_h(y)$. By the double quotient description and the choice of $a_{c'}$, there are $\gamma=(\gamma_{\mathrm{unif}},\gamma')\in G(F)$, written under the first decomposition in \eqref{eq:global-adelic-decompositions}, and $u'\in U'$ such that $\gamma'ah'u'=a_{c'}$. Therefore $t_h(y)$ is also represented by $(\gamma_{\mathrm{unif}}\cdot(\widetilde m\cdot\delta),a_{c'})$. Let $m'_0=(\rho(\gamma_{\mathrm{unif}}\cdot (\widetilde m\cdot\delta)),t_h(y))\in\mathcal M_{c'}(\Cp)$. This is a lift of $t_h(y)$. Let $m'$ be any lift of $t_h(y)$ to $\mathcal M_{c'}$. The fibers of $q_{c'}$ are $\Gamma_{c'}$-orbits, so $m'=\gamma_0\cdot m'_0$ for some $\gamma_0\in\Gamma_{c'}$. Hence
\[
 \begin{aligned}
 \pi_{c'}(m')
  &=\gamma_0\rho (\gamma_{\mathrm{unif}}\cdot
      (\widetilde m\cdot\delta))\\
  &=\gamma_0\overline{\gamma_{\mathrm{unif}}}\,
      \rho(\widetilde m)
   =\gamma_0\overline{\gamma_{\mathrm{unif}}}\,\pi_c(m)
   \in J\,\pi_c(W_c(\Cp)),
 \end{aligned}
\]
where the bar denotes the image in $J$.

The preceding calculation then verifies \eqref{eq:abstract-period containment} for every $h$, every $y\in \operatorname{Sh}_{\mathcal U_h}(\mathbf H^{\mathrm{int}})(\Cp)$ with $s_h(y)\in Z$, and every lift of $t_h(y)$. 

\end{proof}

\subsection{Proof of the main theorem}

\noindent\textbf{Product levels.}
We first apply the criterion at the product levels appearing in the uniformization.

\begin{theorem}
\label{thm:varshavsky-product-level}
Let $\mathcal U$ be any sufficiently small neat product level of the form \eqref{eq:varshavsky-product-level}, let $S=\operatorname{Sh}_{\mathcal U}(\mathbf H^{\mathrm{int}})$, let $\mathcal{S}h=S_{\Cp}$, and let $Z\subset\mathcal{S}h$ be a closed reduced algebraic subvariety. If $\dim Z<\dim S$, then $\Int_{S(\Cp)}\overline{\HO(Z)}^{\,\mathrm{cl}}=\varnothing$.
\end{theorem}

\begin{proof}
Since $S$ is projective and $Z$ is closed, $Z^{\Ber}$ is compact. In the component presentations \eqref{eq:component-uniformization}, the maps $q_c$ are analytic coverings and the maps $\pi_c$ are finite \'etale and therefore open on $\Cp$-points. Thus \textup{(1)--(3)} of \Cref{subsec:criterion-setup} hold. The local hypothesis follows from \Cref{thm:local-drinfeld}. By \Cref{prop:full-period containment}, the full finite adelic Hecke correspondences have the full period containment property for $Z^{\an}$. Since $\dim Z<\dim\mathcal D$, \Cref{thm:global-criterion} applies to $Z^{\an}$ and proves the assertion on all geometric components in \eqref{eq:component-uniformization}.
\end{proof}

\noindent\textbf{Pass to arbitrary neat levels.}
In the following change of level lemma, $\mathbf H$ denotes the reductive group of an arbitrary Shimura datum, as in \Cref{sec:definitions}. We will apply it to $\mathbf H^{\mathrm{int}}$.

\begin{lemma}
\label{lem:change-of-level}
Let $V\triangleleft \mathcal K\subset\mathbf H(\Adele_f)$ be neat compact open subgroups. Let $S_V=\operatorname{Sh}_V(\mathbf H)_{\Cp}$ and $S_{\mathcal K}=\operatorname{Sh}_{\mathcal K}(\mathbf H)_{\Cp}$, and let $\pi:S_V\to S_{\mathcal K}$ be the change of level map. If $Z\subset S_{\mathcal K}$ is closed and reduced and $Z_V=(\pi^{-1}(Z))_{\mathrm{red}}$, then
\begin{equation}\label{eq:hecke-inverse-image}
     \pi^{-1} (\HO_{\mathcal K}(Z))=\HO_V(Z_V).
\end{equation}
If the $p$-adic closure of $\HO_V(Z_V)$ is nowhere dense in $S_V(\Cp)$, then the $p$-adic closure of $\HO_{\mathcal K}(Z)$ is nowhere dense in $S_{\mathcal K}(\Cp)$.
\end{lemma}

\begin{proof}
For $g\in\mathbf H(\Adele_f)$, let $V_g=V\cap gVg^{-1}$ and $\mathcal K_g=\mathcal K\cap g\mathcal K g^{-1}$. Denote the source and target maps at level $V$ by $s_{g,V},t_{g,V}:\operatorname{Sh}_{V_g}(\mathbf H)_{\Cp}\to S_V$, and those at level $\mathcal K$ by $s_{g,\mathcal K},t_{g,\mathcal K}:\operatorname{Sh}_{\mathcal K_g}(\mathbf H)_{\Cp}\to S_{\mathcal K}$. The inclusion $V_g\subset\mathcal K_g$ induces a change of level map $\pi_g:\operatorname{Sh}_{V_g}(\mathbf H)_{\Cp}\to\operatorname{Sh}_{\mathcal K_g}(\mathbf H)_{\Cp}$ satisfying
\[
\pi\circ s_{g,V}=s_{g,\mathcal K}\circ\pi_g,
\quad
\pi\circ t_{g,V}=t_{g,\mathcal K}\circ\pi_g.
\]
The map $\pi_g$ is finite étale and surjective. Indeed, after base change of the canonical models from the reflex field to $\mathbb C$, the change-of-level map sends $[x,a]_{V_g}$ to $[x,a]_{\mathcal K_g}$ and is surjective. Surjectivity therefore holds for the canonical models and after base change to $\Cp$. Since $\Cp$ is algebraically closed, every nonempty finite fiber has a $\Cp$-point. Thus $\pi_g$ is also surjective on $\Cp$-points.

If $y_V\in\operatorname{Sh}_{V_g}(\mathbf H)(\Cp)$ satisfies $s_{g,V}(y_V)\in Z_V(\Cp)$, then $s_{g,\mathcal K}(\pi_g(y_V))=\pi(s_{g,V}(y_V))\in Z(\Cp)$, and $\pi(t_{g,V}(y_V))=t_{g,\mathcal K}(\pi_g(y_V))$. This proves $\pi(\HO_V(Z_V))\subseteq\HO_{\mathcal K}(Z)$. Conversely, let $y\in\operatorname{Sh}_{\mathcal K_g}(\mathbf H)(\Cp)$ satisfy $s_{g,\mathcal K}(y)\in Z(\Cp)$, and choose $y_V$ with $\pi_g(y_V)=y$. Then $\pi(s_{g,V}(y_V))=s_{g,\mathcal K}(y)\in Z(\Cp)$, so $s_{g,V}(y_V)\in Z_V(\Cp)$, while $\pi(t_{g,V}(y_V))=t_{g,\mathcal K}(y)$. As $g$ varies, this gives
\begin{equation}\label{eq:hecke-orbit-image}
    \pi(\HO_V(Z_V))=\HO_{\mathcal K}(Z).
\end{equation}

Since $V$ is normal in $\mathcal K$, right translation defines an action of the finite group $\mathcal K/V$ on $S_V$, and the fibers of $\pi$ are its orbits. For $k\in\mathcal K$, write $R_k$ for the induced automorphism of $S_V$. Normality gives $V_{gk}=V_g$, so the correspondences for $g$ and $gk$ at level $V$ have the same domain, and
\(s_{gk,V}=s_{g,V},\) \(t_{gk,V}=R_k\circ t_{g,V}.\)
Consequently, if $s_{g,V}(y_V)\in Z_V(\Cp)$, then $R_k(t_{g,V}(y_V))=t_{gk,V}(y_V)\in\HO_V(Z_V)$. Hence $\HO_V(Z_V)$ is invariant under $\mathcal K/V$ and contains every fiber of $\pi$ that it meets. Together with \eqref{eq:hecke-orbit-image}, this proves  \eqref{eq:hecke-inverse-image}.

All closures in the remainder of the proof are taken in the $p$-adic topology. For every subset $A\subset S_{\mathcal K}(\Cp)$, continuity gives $\overline{\pi^{-1}(A)}\subseteq\pi^{-1}(\overline A)$. For the reverse inclusion, let $x\in\pi^{-1}(\overline A)$ and let $U$ be an open neighborhood of $x$. Since $\pi$ is finite étale, its map on $\Cp$-points is open. Thus $\pi(U)$ is an open neighborhood of $\pi(x)$ and meets $A$, so $U$ meets $\pi^{-1}(A)$. Therefore $\pi^{-1}(\overline A)=\overline{\pi^{-1}(A)}$. Taking $A=\HO_{\mathcal K}(Z)$ and using \eqref{eq:hecke-inverse-image} gives
\[
\pi^{-1} (\overline{\HO_{\mathcal K}(Z)})
=
\overline{\HO_V(Z_V)}.
\]
Suppose that $\overline{\HO_{\mathcal K}(Z)}$ contains a nonempty open subset $U$. Since $\pi$ is continuous and surjective on $\Cp$-points, $\pi^{-1}(U)$ is a nonempty open subset of $S_V(\Cp)$ contained in $\overline{\HO_V(Z_V)}$, contrary to the assumed nowhere density.
\end{proof}

\begin{lemma}
\label{lem:product-level-refinement}
Every compact open subgroup $\mathcal K\subset\mathbf H^{\mathrm{int}}(\Adele_f)$ contains a normal compact open subgroup of the form \(\mathcal V=\left(\prod_{i=1}^rT_{\ell,i}\right)\times V'\) under \eqref{eq:global-adelic-decompositions}, where $V'\subset G'$ is compact open. The integer $\ell$ can be taken arbitrarily large. If $\mathcal K$ is neat, then $\mathcal V$ is neat.
\end{lemma}

\begin{proof}
Since $\mathcal K$ is open, it contains $(\prod_iT_{\ell_0,i})\times V'_0$ for some $\ell_0$ and some compact open subgroup $V'_0\subset G'$. Let $\operatorname{pr}_{G'}$ be the projection to $G'$ in the second decomposition of \eqref{eq:global-adelic-decompositions}. Let
\(
V'=\bigcap_{a\in\operatorname{pr}_{G'}(\mathcal K)}aV'_0a^{-1}.
\)
Since $V'_0$ has finite index in $\operatorname{pr}_{G'}(\mathcal K)$, only finitely many distinct conjugates occur in this intersection. Thus $V' \subset V'_0$ is a compact open subgroup of $G'$ which is normal in $\operatorname{pr}_{G'}(\mathcal K)$. 

Since the unique maximal order $\mathcal{O}_{\widetilde D_i}$ and every power of its Jacobson radical are preserved by conjugation by $\widetilde D_i^\times$, the principal congruence subgroup $T_{\ell,i}$ is normal in $\widetilde D_i^\times$. It follows that for every $\ell\geq\ell_0$, the subgroup $\mathcal V=(\prod_iT_{\ell,i})\times V'$ is contained in $\mathcal K$ and normal in $\mathcal K$. If $\mathcal K$ is neat, so is $\mathcal V$.
\end{proof}

The two lemmas together give the arbitrary level statement of \Cref{thm:main-intro}.

\begin{theorem}
\label{thm:varshavsky-main}
Let $\mathcal K\subset\mathbf H^{\mathrm{int}}(\Adele_f)$ be any sufficiently small neat compact open subgroup, put $S_{\mathcal K}=\operatorname{Sh}_{\mathcal K}(\mathbf H^{\mathrm{int}})_{\Cp}$, and let $Z\subset S_{\mathcal K}$ be a closed reduced algebraic subvariety. If $\dim Z<\dim S_{\mathcal K}$, then the $p$-adic closure of $\HO_{\mathcal K}(Z)$ is nowhere dense in $S_{\mathcal K}(\Cp)$.
\end{theorem}

\begin{proof}
Choose a normal product level $\mathcal V\subset\mathcal K$ as in \Cref{lem:product-level-refinement}, put $S_{\mathcal V}=\operatorname{Sh}_{\mathcal V}(\mathbf H^{\mathrm{int}})_{\Cp}$, and let $\pi:S_{\mathcal V}\to S_{\mathcal K}$ be the change of level map. Put $Z_{\mathcal V}=(\pi^{-1}(Z))_{\mathrm{red}}$. Since $\pi$ is finite \'etale and surjective, $\dim Z_{\mathcal V}=\dim Z$ and $\dim S_{\mathcal V}=\dim S_{\mathcal K}$. \Cref{thm:varshavsky-product-level} shows that the $p$-adic closure of $\HO_{\mathcal V}(Z_{\mathcal V})$ is nowhere dense in $S_{\mathcal V}(\Cp)$. The result then follows from \Cref{lem:change-of-level}.
\end{proof}

\begin{remark}\label{rmk:Q-simple}
The adjoint group $\Res_{F/\mathbb Q}\PGU(D^{\mathrm{int}},\alpha^{\mathrm{int}})$ is $\mathbb Q$-simple. Indeed, after base change to $\overline{\mathbb Q}$, it is a product of copies of $\PGL_d$ indexed by the embeddings $F\hookrightarrow\overline{\mathbb Q}$. Every connected normal subgroup is a product of a subset of these factors. If the subgroup is defined over $\mathbb Q$, this subset is Galois invariant. Since the Galois action on the embeddings of $F$ is transitive, the subset is either empty or the whole set. In particular, when $r>1$, the period domain is a product, but the adjoint Shimura datum does not decompose as a nontrivial product.
\end{remark}

The main theorem applies only at places satisfying Varshavsky's uniformization hypotheses. However, for any prescribed rational prime $p$, one can choose a unitary Shimura datum of dimension $d-1$ for each $d\geq2$ so that these hypotheses hold at a place above $p$.
\begin{proposition}\label{prop:one-place-existence}
For every rational prime $p$ and every $d\geq2$, there is a Varshavsky datum of degree $d$ with $r=1$, depending on $p$, whose Shimura tower is globally Drinfeld-uniformized at a place above $p$.
\end{proposition}

\begin{proof}
Let $K=\mathbb Q(\sqrt{1-4p})$. Write $p\mathcal O_K=w\overline w$. The polynomial $T^2-T+p$ has negative discriminant and reduces modulo $p$ to $T(T-1)$. Hensel's lemma therefore gives $K\otimes_{\mathbb Q}\mathbb Q_p\simeq\mathbb Q_p\times\mathbb Q_p$, so $p$ splits in $K$. Set $F=\mathbb Q$, take $D=M_d(K)$, and equip $D$ with the conjugate-transpose involution $\alpha$. Then $\alpha$ is positive definite at the real place, and $D$ is split at $w$.

By the Brauer--Hasse--Noether theorem, as in \cite[\S2.1.6]{VarshavskyI}, there is a central division algebra $D^{\mathrm{int}}$ of degree $d$ over $K$ with local invariants $1/d$ at $w$, $-1/d$ at $\overline w$, and zero at every other place. In particular, $D^{\mathrm{int}}$ is locally isomorphic to $D=M_d(K)$ away from $w$ and $\overline w$. By \cite[Proposition~2.1.7]{VarshavskyI}, it has an involution $\alpha^{\mathrm{int}}$ of the second kind with signature $(d-1,1)$ at the real place. The reflex field is $K$ by \cite[Proposition~3.1.3]{VarshavskyI} so we take $v=w$ and fix an embedding $K_w\hookrightarrow\Cp$. The distinctness condition in \cite[\S2.5]{VarshavskyII} is automatic because $r=1$; its splitting and Brauer-invariant conditions hold by construction. The required equivariant global uniformization now follows from \cite[Theorem~2.13]{VarshavskyII}.
\end{proof}

\section{A PEL datum with the same adjoint datum}
\label{subsec:pel-central-lift}

Let $\mathbb S=\Res_{\mathbb C/\mathbb R}\mathbf G_m$, and let $h:\mathbb S\to\mathbf H^{\mathrm{int}}_{\mathbb R}$ be Varshavsky's homomorphism. By \cite[\S2.2]{VarshavskyII}, its restriction to $\mathbb R^\times$ is trivial. This datum is not of Hodge type, since the restriction of the homomorphism in the Siegel datum to $\mathbb R^\times$ is the nontrivial scalar weight.

Let $\nu:\mathbf H^{\mathrm{int}}\to \Res_{F/\mathbb Q}\mathbf G_m$ be the similitude character, and let $\mathbf U=\ker\nu$. Regard $\mathbf G_m$ as the diagonal subtorus of $\Res_{F/\mathbb Q}\mathbf G_m$, and define
\[
 \mathbf H_0=\mathbf H^{\mathrm{int}}
 \times_{\Res_{F/\mathbb Q}\mathbf G_m}\mathbf G_m,
 \quad
 \mathbf T=\Res_{K/\mathbb Q}\mathbf G_m
 \times_{\Res_{F/\mathbb Q}\mathbf G_m}\mathbf G_m,
\]
where the map from $\Res_{K/\mathbb Q}\mathbf G_m$ is induced by $N_{K/F}$. Both groups have a multiplier with values in $\mathbf G_m$, and scalar multiplication identifies $\mathbf T$ with a central subgroup of $\mathbf H_0$. Define $\mathbf H^\sharp=\mathbf T\times_{\mathbf G_m}\mathbf H_0$ using these two multiplier characters.

At every real place, choose the embedding $K \hookrightarrow \mathbb C$ used to define $h$, and let $\eta:\mathbb S\to\mathbf T_{\mathbb R}$ be defined by \(\eta(z)= (\overline z^{-1},\ldots,\overline z^{-1})\) for $z \in \mathbb S(\mathbb R)$. Let $h_0=\eta h$. For a real embedding $\tau:F\hookrightarrow\mathbb R$, the corresponding component of $h_0$ is
\[
 h_{0,\tau}(z)=
 \begin{cases}
  \operatorname{diag}(\bar z^{-1},\ldots,\bar z^{-1},z^{-1}),
       &\tau\text{ has signature }(d-1,1),\\
 \bar z^{-1}I_d,&\tau\text{ definite}.
 \end{cases}
\]
Here the formula for $h$ is the one in \cite[\S2.2]{VarshavskyII}. Both $\eta$ and $h_0$ have multiplier $(z\bar z)^{-1}$, so $h^\sharp=(\eta,h_0)$ takes values in $\mathbf H^\sharp_{\mathbb R}$. Let $\mathcal X_0$ and $\mathcal X_{\mathbf T}$ be the conjugacy classes of $h_0$ and $\eta$.

\begin{proposition}\label{prop:pel-datum}
Let $\mathcal X^\sharp$ and $\mathcal X$ be the conjugacy classes of $h^\sharp$ and $h$, respectively. Then $(\mathbf H^\sharp,\mathcal X^\sharp)$ is a PEL Shimura datum. The map $q:\mathbf H^\sharp\to\mathbf H^{\mathrm{int}}$ given by $q(t,g)=t^{-1}g$ is a morphism of Shimura data, and it induces an isomorphism $(\mathbf H^\sharp)^{\mathrm{ad}}\simeq (\mathbf H^{\mathrm{int}})^{\mathrm{ad}}$.
\end{proposition}

\begin{proof}
We first justify the PEL assertion. Let $(D^{\mathrm{int}})^- = \{c\in D^{\mathrm{int}}:\alpha^{\mathrm{int}}(c)=-c\}$. At the real places, a direct computation shows that $c_\infty=h_{0}(i)$ belongs to \( ((D^{\mathrm{int}})^-\otimes_{\mathbb Q}\mathbb R) \cap  (D^{\mathrm{int}}\otimes_{\mathbb Q}\mathbb R)^\times.\) Write $\Int(c)(a)=cac^{-1}$ for inner conjugation by an invertible element $c$. The involution $\Int(c_\infty)\circ\alpha^{\mathrm{int}}$ is positive, and the regular-module construction defined by $c_\infty$ gives a polarization for which $h_0$ is positive; see \cite[Lemma~4.1(3) and the PEL construction on p.~387]{KottwitzPoints}. Both positivity conditions and invertibility remain valid under sufficiently small perturbations of $c_\infty$ within $(D^{\mathrm{int}})^-\otimes_{\mathbb Q}\mathbb R$. Since $(D^{\mathrm{int}})^-$ is a $\mathbb Q$-vector space, choose $c\in (D^{\mathrm{int}})^-\cap (D^{\mathrm{int}})^\times$ sufficiently close to $c_\infty$ at every real place. Then $\beta=\Int(c)\circ\alpha^{\mathrm{int}}$ is a positive involution.

Let $V$ be the underlying $\mathbb Q$-vector space of $D^{\mathrm{int}}$, and let $B=(D^{\mathrm{int}})^{\mathrm{op}}$ act on $V$ by right multiplication. Give $B$ the involution $\beta$, and let $\psi(x,y)=\operatorname{Tr}_{K/\mathbb Q}\operatorname{Trd}_{D^{\mathrm{int}}/K}(xc\alpha^{\mathrm{int}}(y))$. The form $\psi$ is nondegenerate and alternating. Cyclicity of the reduced trace gives $\psi(ax,y)=\psi(x,\alpha^{\mathrm{int}}(a)y)$ and $\psi(xb,y)=\psi(x,y\beta(b))$ for $a,b\in D^{\mathrm{int}}$. Thus $(V,\psi)$ is a symplectic $(B,\beta)$-module, $\operatorname{End}_B(V)=D^{\mathrm{int}}$ by left multiplication, and the involution adjoint to $\psi$ restricts to $\alpha^{\mathrm{int}}$ on $\operatorname{End}_B(V)=D^{\mathrm{int}}$. For every $\mathbb Q$-algebra $R$, the group of $B$-linear similitudes of $(V,\psi)$ is $\{a\in(D^{\mathrm{int}}\otimes_{\mathbb Q}R)^\times:\alpha^{\mathrm{int}}(a)a\in R^\times\}=\mathbf H_0(R)$. By the choice of $c$, the symmetric form $(x,y)\mapsto\psi(x,h_0(i)y)$ is positive definite on $V\otimes_{\mathbb Q}\mathbb R$. Hence $(\mathbf H_0,\mathcal X_0)$ is the PEL datum of \cite[\S5, p.~389]{KottwitzPoints}; see also \cite[Propositions~8.12 and~8.14 and Definition~8.15]{MilneShimura}.

For the torus $\mathbf T$, we view $K$ as a $\mathbb Q$-vector space with the action of itself by multiplication and its complex-conjugation involution. Choose $\xi\in K^\times$ with $\overline\xi=-\xi$ and positive imaginary part at each of the embeddings chosen to define $\eta$; such an element exists by weak approximation. Set $\psi_{\mathbf T}(x,y)=\operatorname{Tr}_{K/\mathbb Q}(\xi x\overline y)$. This form is nondegenerate and alternating, and the group of $K$-linear similitudes is $\mathbf T$, with multiplier $N_{K/F}$. At every chosen embedding, $\eta(i)$ acts by $i$, so $\psi_{\mathbf T}(x,\eta(i)x)>0$ for every nonzero $x\in K\otimes_{\mathbb Q}\mathbb R$. Thus $(\mathbf T,\mathcal X_{\mathbf T})$ is a rank-one CM PEL datum.

Now take $K\oplus V$ with alternating form $\psi_{\mathbf T}\oplus\psi$ and the componentwise action of $K\times B$, equipped with the involution $(a,b)\mapsto(\overline a,\beta(b))$. A similitude commuting with this algebra preserves the two summands, because it commutes with the two central idempotents \((1,0)\) and \((0,1)\). Its restrictions to the summands have the same multiplier, so the group of these similitudes is $\mathbf T\times_{\mathbf G_m}\mathbf H_0=\mathbf H^\sharp$. The homomorphism $h^\sharp=(\eta,h_0)$ has the required Hodge types, and the associated symmetric form is positive definite on both summands. This proves that $(\mathbf H^\sharp,\mathcal X^\sharp)$ is of PEL type. We note that this fiber product construction is also used in \cite[\S3.2, (3.3)]{RSZUnitary}.

For $(t,g)\in\mathbf H^\sharp$, equality of the multipliers gives $\nu(t^{-1}g)=1$, and every $u\in\mathbf U$ is the image of $(1,u)$. Moreover, $q\circ h^\sharp=\eta^{-1}h_0=h$. The map $(t,g)\mapsto(t,t^{-1}g)$ identifies $\mathbf H^\sharp$ with $\mathbf T\times\mathbf U$, and under this identification $\ker q$ is the first factor. The group $\mathbf U$ contains the derived group of $\mathbf H^{\mathrm{int}}$ and maps onto $(\mathbf H^{\mathrm{int}})^{\mathrm{ad}}$. Moreover, $\ker(\mathbf U\to(\mathbf H^{\mathrm{int}})^{\mathrm{ad}})=\mathbf U\cap Z(\mathbf H^{\mathrm{int}})=Z(\mathbf U)$. It follows that $(\mathbf H^\sharp)^{\mathrm{ad}}\simeq\mathbf U^{\mathrm{ad}}\simeq(\mathbf H^{\mathrm{int}})^{\mathrm{ad}}$.
\end{proof}

In particular, Varshavsky's datum is of abelian type.

To transfer the nowhere density to the PEL Shimura varieties, we first describe the maps induced by $q$ at finite level.
\begin{lemma}
\label{lem:pel-central-lift-finite}
 Let $E^\sharp$ be the reflex field of $(\mathbf H^\sharp,\mathcal X^\sharp)$. Let $EE^\sharp\hookrightarrow\Cp$ be an embedding whose restriction to $E$ is the fixed embedding $E\hookrightarrow E_v\hookrightarrow\Cp$. Let $\mathcal U$ be the product level in \eqref{eq:varshavsky-product-level}. There are arbitrarily small neat compact open subgroups $\mathcal U^\sharp\subset\mathbf H^\sharp(\Adele_f)$ with $q(\mathcal U^\sharp)\subset\mathcal U$. For every such subgroup, let $S^\sharp=\operatorname{Sh}_{\mathcal U^\sharp}(\mathbf H^\sharp,\mathcal X^\sharp)_{\Cp}$. The map $f:S^\sharp\to \mathcal{S}h$ induced by $q$ is finite \'etale, and its image is an open and closed union of geometric components.
\end{lemma}

\begin{proof}
Since $q$ is continuous on finite adelic points, $q^{-1}(\mathcal U)$ is an open neighborhood of the identity and contains arbitrarily small neat compact open subgroups. Fix such a subgroup $\mathcal U^\sharp$. By \cite[Theorem~5.16 and Remark~13.8]{MilneShimura}, the morphism of Shimura data induces a morphism of canonical models over $EE^\sharp$. We first prove finiteness and \'etaleness for these models and then base change to $\Cp$ through the fixed embedding $EE^\sharp\hookrightarrow\Cp$.

Choose a connected component $(\mathcal X^\sharp)^+$ of $\mathcal X^\sharp$, and let $\mathcal X^+$ be its image in $\mathcal X$. Since $q$ induces an isomorphism of adjoint Shimura data, it identifies $(\mathcal X^\sharp)^+$ with $\mathcal X^+$. Choose $a\in\mathbf H^\sharp(\Adele_f)$ representing a connected component of the source Shimura variety. The image component is represented by $q(a)$. Let $\Gamma^\sharp$ and $\Gamma$ be the images of \(\mathbf H^\sharp(\mathbb Q)^+\cap a\mathcal U^\sharp a^{-1}\) and \(\mathbf H^{\mathrm{int}}(\mathbb Q)^+\cap q(a)\mathcal Uq(a)^{-1}\) in $(\mathbf H^{\mathrm{int}})^{\mathrm{ad}}(\mathbb Q)$, respectively. Here the superscript $+$ denotes the subgroup preserving the chosen connected component of the corresponding symmetric domain. The inclusion $q(\mathcal U^\sharp)\subset\mathcal U$ gives $\Gamma^\sharp\subset\Gamma$.

By \cite[Lemma~5.13]{MilneShimura}, the map on these connected components is
\[
 \Gamma^\sharp\backslash\mathcal X^+
 \longrightarrow
 \Gamma\backslash\mathcal X^+,
 \qquad [x]_{\Gamma^\sharp}\longmapsto[x]_{\Gamma}.
\]
By \cite[Proposition~3.2 and the discussion following Lemma~5.13]{MilneShimura}, both groups are arithmetic subgroups of the same adjoint group. They are therefore commensurable, so $[\Gamma:\Gamma^\sharp]<\infty$. Neatness makes their actions on $\mathcal X^+$ free. The displayed map is consequently a finite holomorphic covering so in particular a local biholomorphism.

Every target component is compact because the target Shimura variety is projective. Each source component is a finite covering of a target component, so it is compact as well. There are only finitely many source components by \cite[Lemma~5.12]{MilneShimura}. The source Shimura variety thus has compact complex analytification. Since it is quasi-projective, it is projective. The algebraic morphism is therefore proper, and the componentwise description shows that it has finite fibers. Hence it is finite. The same description shows that it is \'etale, since its analytification is a local biholomorphism.

Finiteness and \'etaleness descend from $\mathbb C$ to $EE^\sharp$ and are preserved by base change to $\Cp$. Thus $f:S^\sharp\to\mathcal{S}h$ is finite \'etale. A finite morphism is closed and an \'etale morphism is open, so its image is an open and closed union of geometric components.
\end{proof}

\begin{proof}[Proof of \Cref{cor:pel-intro}]
Let $f:S^\sharp\to\mathcal{S}h$ be the map in \Cref{lem:pel-central-lift-finite}, and let $Z_0=(f(Z^\sharp))_{\mathrm{red}}$. Since $f$ is finite, $Z_0$ is closed and $\dim Z_0=\dim Z^\sharp$. The two ambient Shimura varieties have the same dimension because their adjoint Shimura data are isomorphic. Hence $\dim Z_0<\dim\mathcal{S}h$. By \Cref{thm:varshavsky-main}, the $p$-adic closure of $\HO(Z_0)$ is nowhere dense in $\mathcal{S}h(\Cp)$.

Fix $a\in\mathbf H^\sharp(\Adele_f)$, and put $\mathcal U^\sharp_a=\mathcal U^\sharp\cap a\mathcal U^\sharp a^{-1}$ and $\mathcal U_{q(a)}=\mathcal U\cap q(a)\mathcal Uq(a)^{-1}$. Define
\[
 S^\sharp_a=\operatorname{Sh}_{\mathcal U^\sharp_a} (\mathbf H^\sharp,\mathcal X^\sharp)_{\Cp},\quad S_{q(a)}=\operatorname{Sh}_{\mathcal U_{q(a)}}(\mathbf H^{\mathrm{int}},\mathcal X)_{\Cp}.
\]
Write $s^\sharp_a,t^\sharp_a:S^\sharp_a\to S^\sharp$ and $s_{q(a)},t_{q(a)}:S_{q(a)}\to\mathcal{S}h$ for the corresponding Hecke maps. The maps \(s^\sharp_a\) and \(s_{q(a)}\) forget part of the level structure, and the maps \(t^\sharp_a\) and \(t_{q(a)}\) are induced by right multiplication by $a$ and $q(a)$, respectively.

The inclusion $q(\mathcal U^\sharp)\subset\mathcal U$ implies $q(\mathcal U^\sharp_a)\subset\mathcal U_{q(a)}$. Thus $q$ induces a morphism $f_a: S^\sharp_a \to S_{q(a)}$. Functoriality gives
\[
 s_{q(a)}\circ f_a=f\circ s^\sharp_a,
 \quad
 t_{q(a)}\circ f_a=f\circ t^\sharp_a.
\]
If $y\in S^\sharp_a(\Cp)$ satisfies $s^\sharp_a(y)\in Z^\sharp(\Cp)$, then $s_{q(a)}(f_a(y))=f(s^\sharp_a(y))\in Z_0(\Cp)$. It follows that $f(t^\sharp_a(y))=t_{q(a)}(f_a(y))\in\HO(Z_0)$. Taking the union over all $a\in\mathbf H^\sharp(\Adele_f)$ gives $f(\HO(Z^\sharp))\subseteq\HO(Z_0)$.

 By continuity of $f$, the set $f^{-1}(\overline{\HO(Z_0)}^{\,\mathrm{cl}})$ is closed and contains $\HO(Z^\sharp)$. Therefore
\[
 \overline{\HO(Z^\sharp)}^{\,\mathrm{cl}}
 \subseteq f^{-1} (\overline{\HO(Z_0)}^{\,\mathrm{cl}}).
\]
Since $f$ is \'etale, its map on $\Cp$-points is locally an analytic isomorphism and is therefore open. If $f^{-1}(\overline{\HO(Z_0)}^{\,\mathrm{cl}})$ contained a nonempty open subset $U\subset S^\sharp(\Cp)$, then $f(U)$ would be a nonempty open subset of $\mathcal{S}h(\Cp)$ contained in $\overline{\HO(Z_0)}^{\,\mathrm{cl}}$, contradicting its nowhere density. Thus \(f^{-1} (\overline{\HO(Z_0)}^{\,\mathrm{cl}})\) is nowhere dense, and so is $\overline{\HO(Z^\sharp)}^{\,\mathrm{cl}}$.
\end{proof}

\section{Comparison and remarks on the general case}
\label{sec:scope}

\subsection{Comparison with Maulik--Poonen}
\label{subsec:maulik-poonen}
In this subsection, we assume that the Shimura variety has a smooth proper integral model over which the associated abelian scheme extends. Let $L/\Qp$ be finite, and let $\mathscr S$ be an irreducible smooth proper scheme over $\mathcal O_L$ carrying an abelian scheme $\mathcal A\to\mathscr S$. Let $S=\mathscr S_L$, write $\mathcal A_L=\mathcal A\times_{\mathcal O_L}L$, and let $\bar\eta$ be a geometric generic point of $S$. For $s\in S(\Cp)$, write $\mathcal A_s$ for the fiber of $\mathcal A_L$ over $s$ which is an abelian variety over $\Cp$. Properness identifies $\mathscr S(\mathcal O_{\Cp})$ with $S(\Cp)$. Under this identification, Maulik--Poonen prove that the set of points where the specialization map $\operatorname{End}(\mathcal A_{\bar\eta})\to\operatorname{End}(\mathcal A_s)$ is not an isomorphism is nowhere dense in the $p$-adic topology \cite[Setup~1.6 and Proposition~1.13]{MaulikPoonen}.

For the case where the subvariety $Z$ is a point, suppose in addition that $S$ is a geometrically connected, positive-dimensional component of a Shimura variety of Hodge type. Fix a Hodge embedding and a compatible neat Siegel level. Write $\pi^{\mathrm{univ}}:\mathfrak A_g\to\mathcal A_g$ for the universal abelian scheme over the corresponding Siegel moduli space, omitting the fixed level from the notation. Let $\iota:S\to\mathcal A_{g,L}$ be the induced morphism, and assume that $\mathcal A_L\simeq\iota^*\mathfrak A_{g,L}$. Let $\mathscr T=\mathscr S\times_{\mathcal O_L}\mathscr S$, and consider the abelian scheme $\operatorname{pr}_1^*\mathcal A\times_{\mathscr T} \operatorname{pr}_2^*\mathcal A$ over $\mathscr T$, where $\operatorname{pr}_1,\operatorname{pr}_2:\mathscr T\to\mathscr S$ are the two projections. The scheme $\mathscr T$ is irreducible, since it is smooth over $\mathcal O_L$ and its generic fiber is geometrically integral.

The two factors of the geometric generic fiber of this product family are not isogenous. Indeed, an isogeny between them would extend over a finite cover of a nonempty open subset of $S_{\overline L}\times_{\overline L}S_{\overline L}$. Choose an $\overline L$-point in the first projection of this open subset. Its fiber under the first projection is a nonempty open subset of $S_{\overline L}$, and the finite cover restricts to a finite surjective cover of this fiber. On this cover, the first abelian scheme is constant, so the geometric generic fiber of the second, which is the pullback of $\mathcal A_{\overline L}$, is isogenous to a constant abelian variety. By \cite[Remark~1.16]{MaulikPoonen}, the abelian scheme $\mathcal A_{\overline L}\to S_{\overline L}$ would then become constant after a finite \'etale base change. This contradicts the positive-dimensional image of $S$ under the chosen morphism to the Siegel moduli space.

Let $N\subset S(\Cp)\times S(\Cp)$ be the endomorphism jumping locus of the product family, relative to its geometric generic fiber over $\mathscr T$. By \cite[Setup~1.6 and Proposition~1.13]{MaulikPoonen} and properness of $\mathscr T$, the set $N$ is nowhere dense in the $p$-adic topology. Every pair $(s,t)$ for which $\mathcal A_s$ and $\mathcal A_t$ are isogenous belongs to $N$. To see this, suppose that the specialization map on endomorphism rings were an isomorphism at $(s,t)$. After tensoring with $\mathbb Q$, an isogeny between $\mathcal A_s$ and $\mathcal A_t$ and its inverse in $\operatorname{Hom}\otimes\mathbb Q$ would lift to rational endomorphisms of the geometric generic product. Specialization respects the projections onto the two factors. Its injectivity therefore shows that the lifted maps give mutually inverse rational homomorphisms between the generic factors, contradiction.

Now fix any $x\in S(\Cp)$, and let $I_x$ be the set of $s\in S(\Cp)$ for which $\mathcal A_s$ is isogenous to $\mathcal A_x$. Then $I_x\times I_x\subseteq N$. If the $p$-adic closure of $I_x$ contained a nonempty open subset $U$, the set $(I_x\cap U)\times(I_x\cap U)$ would be dense in $U\times U$. The $p$-adic closure of $N$ would therefore contain $U\times U$, contradiction. Thus the $p$-adic closure of $I_x$ is nowhere dense.

Hecke correspondences induce quasi-isogenies, so every Hecke translate of $x$ lying in $S$ belongs to $I_x$. Suppose that over a common finite extension of $L$, every geometric component of the fixed-level Shimura variety admits a smooth proper model carrying the induced abelian scheme. On each component, any two points of the Hecke orbit again have isogenous fibers, so the same product argument applies to the intersection of the orbit with that component. Since there are finitely many components, the $p$-adic closure of the full finite adelic Hecke orbit of $x$ is nowhere dense. We thank Ananth Shankar for explaining this consequence of Maulik--Poonen.

There is a related argument for positive-dimensional subvarieties with generically extra endomorphisms. Let $Z\subset S_{\Cp}$ be a closed integral subvariety, and let $\bar\eta_Z$ be a geometric generic point of $Z$. Write $\mathcal A_{\bar\eta_Z}$ for the fiber of $(\mathcal A_L)_{\Cp}$ at $\bar\eta_Z$. Suppose that $\dim_{\mathbb Q}\operatorname{End}^0(\mathcal A_{\bar\eta_Z})>\dim_{\mathbb Q}\operatorname{End}^0(\mathcal A_{\bar\eta})$. Choose a $\mathbb Q$-basis of this endomorphism algebra. After a finite extension of $\Cp(Z)$, all the chosen basis elements are defined. After multiplying them by nonzero integers they are endomorphisms. Let $Z'$ be the normalization of $Z$ in this extension. The map $Z'\to Z$ is finite and surjective, and every point of $Z(\Cp)$ lifts to a point of $Z'(\Cp)$. Pull back $(\mathcal A_L)_{\Cp}$ along $Z'\to Z\hookrightarrow S_{\Cp}$. The chosen endomorphisms of its generic fiber extend to this abelian scheme over $Z'$ by \cite[Chapter~I, Proposition~2.7]{FaltingsChai}. Their specializations on every geometric fiber remain $\mathbb Q$-linearly independent. Indeed, after clearing denominators, a relation on one fiber would give an endomorphism of the abelian scheme over $Z'$ which vanishes on that fiber. Rigidity \cite[Chapter~6, Proposition~6.1(3)]{MumfordFogartyKirwan} would then make it vanish identically, contrary to linear independence on the geometric generic fiber. Thus every point of $Z(\Cp)$ lies in the Maulik--Poonen jumping locus.
A Hecke correspondence induces a quasi-isogeny between the corresponding abelian varieties and therefore preserves the dimension of their rational endomorphism algebras. Thus every point of $\HO(Z)\cap S(\Cp)$ has endomorphism dimension greater than $\dim_{\mathbb Q}\operatorname{End}^0(\mathcal A_{\bar\eta})$ and belongs to the Maulik--Poonen jumping locus. Its $p$-adic closure is therefore nowhere dense in $S(\Cp)$. If every geometric component of the fixed-level Shimura variety admits a smooth proper integral model carrying the induced abelian scheme, and the geometric generic fiber on each component has endomorphism dimension $\dim_{\mathbb Q}\operatorname{End}^0(\mathcal A_{\bar\eta})$, the same argument applies on every component. Since there are finitely many components, the $p$-adic closure of the full finite adelic Hecke orbit of $Z$ is nowhere dense in the entire Shimura variety.

In the good reduction setting above, our application of Maulik--Poonen's result to positive-dimensional $Z$ requires $\dim_{\mathbb Q}\operatorname{End}^0(\mathcal A_{\bar\eta_Z})>\dim_{\mathbb Q}\operatorname{End}^0(\mathcal A_{\bar\eta})$. In contrast, \Cref{thm:varshavsky-main} applies to every closed reduced subvariety of positive codimension under Varshavsky's global uniformization hypotheses. Also, Ananth Shankar and Philip Engel have informed the author that they have a proof of the conjecture when $Z$ is a Shimura subvariety based on Maulik--Poonen's method.

\subsection{General Shimura varieties}
\label{subsec:scope-criterion}

For a Hodge type Shimura datum whose group is unramified at an odd prime $p$, Kim's uniformization describes the formal neighborhood associated with a fixed isogeny class $\phi$ in the special fiber \cite[Theorem~4.7]{KimUniformization}. Its rigid generic fiber is a tube or union of tubes, which we denote by $U_\phi$ \cite[\S5.3 and Theorem~5.4]{KimUniformization}. 

\medskip
\noindent\textbf{Finite quasi-compact lifts.}
 Even when $Z$ is proper over $\Cp$, the intersection $Z^{\an}\cap U_\phi$ need not be quasi-compact. Properness of $Z$ alone therefore does not allow us to apply \Cref{lem:finite-qc-lifts} to this intersection.

Finite quasi-compact lifts can be replaced by the uniform reduction condition in \Cref{prop:uniform-period-reduction}. This condition must control the periods of all lifts and their $J$-translates simultaneously. Kim's uniformization does not establish this condition for lifts of an arbitrary algebraic $Z$.

\medskip
\noindent\textbf{Local actions and period maps.}
Properness of the $J$-action on $\mathcal D^{\Ber}$ is used in the proof of the Drinfeld theorem \Cref{thm:local-drinfeld}; it is not an additional hypothesis of the general criterion once the required local orbit-closure statement is known. The descent in \Cref{lem:pullback-quotient} only requires the period maps to be analytic and open on $\Cp$-points, not finite. Kim constructs a $J_b(\Qp)$-equivariant étale period map to the associated flag variety \cite[\S7.5, Proposition~7.5.2 and Remark~7.5.5]{KimRZSpaces}, where $J_b(\Qp)$ is the group of tensor preserving self quasi-isogenies of the framing $p$-divisible group. After base change to $\Cp$, this map is open on $\Cp$-points. To apply \Cref{thm:global-criterion}, one still needs a local nowhere density theorem for the relevant orbit closures on this flag variety, and étaleness itself does not give such a theorem.

\medskip
\noindent\textbf{Full adelic compatibility.}
Both Kim's and Varshavsky's uniformizations are equivariant for the full finite adelic action on their respective towers \cite[Theorem~5.4]{KimUniformization}. A Hecke operation at $p$ changes the $p$-divisible group and its framing quasi-isogeny compatibly, so the rational Hodge filtration transported to the framing isocrystal is unchanged \cite[\S\S7.4--7.5]{KimRZSpaces}. Hecke operations away from $p$ change the adelic coordinate, while a change of representative in the arithmetic quotient changes the period through $J_b(\Qp)$. These facts give the pointwise period containment within $U_\phi$. The Hecke correspondences preserve $U_\phi$ at fixed level, while the full adelic group acts on the tower.

The remaining difficulty is to place all these periods in one closed nowhere dense set, using either finite quasi-compact lifts or an additional condition such as the uniform reduction condition above. Moreover, an arbitrary $Z$ may meet tubes associated with different isogeny classes. Proving nowhere density on each such tube does not by itself control the closure of their union.

\medskip
If one uses a compactification to treat a subvariety $Z$ that is not proper over $\Cp$, additional questions arise concerning the extension and biextension data at the semiabelian boundary \cite[Chapter~II, \S\S1--2, and Chapter~III]{FaltingsChai}. Consequently, the argument given here does not directly prove the conjecture for an arbitrary PEL or Hodge type Shimura variety.

\end{document}